\documentclass[12pt,a4paper]{amsart}

\usepackage{amssymb}
\usepackage[all]{xy}

\newtheorem{theorem}{Theorem}[section]

\newtheorem{lemma}[theorem]{Lemma}
\newtheorem{prop}[theorem]{Proposition}
\newtheorem{cor}[theorem]{Corollary}

\theoremstyle{remark}
\newtheorem{example}[theorem]{Example}
\newtheorem{remark}[theorem]{Remark}

\newtheorem{examples}[theorem]{Examples}

\newtheorem*{problem*}{Problem}
\newtheorem*{remark*}{Remark}
\newtheorem*{convention*}{Convention}
\newtheorem*{notation*}{Notation}
\newtheorem*{examples*}{Examples}
\newtheorem*{example*}{Example}
\newtheorem*{warning*}{Warning}

\numberwithin{equation}{section}

\def\N{{\mathbb N}}

\def\R{{\mathbb R}}
\def\T{{\mathbb T}}
\def\C{{\mathbb C}}
\def\Z{{\mathbb Z}}
\def\CR{{\mathcal R}}
\def\CS{{\mathcal S}}
\def\CF{{\mathcal F}}
\def\H{{\mathcal H}}

\newcommand{\newspan}{\operatorname{span}}
\newcommand{\clsp}{\overline{\operatorname{span}}}
\newcommand{\End}{\operatorname{End}}
\newcommand{\BS}{\operatorname{BS}}

\newcommand{\Real}{\operatorname{Re}}\newcommand{\Imag}{\operatorname{Im}}

\begin{document}

\title[Direct limits and wavelets]{Direct limits, multiresolution analyses,\\ and wavelets}

\author[Baggett]{Lawrence~W.~Baggett}
\address{Lawrence Baggett, Judith Packer and Arlan Ramsay, Department of Mathematics, University of Colorado, Boulder, Colorado 80309, USA}
\email{baggett, packer and ramsay@euclid.colorado.edu}
\author[Larsen]{Nadia~S.~Larsen}
\address{Nadia Larsen, Mathematics Institute, University of Oslo, Blindern, NO-0316 Oslo, Norway}
\email{nadiasl@math.uio.no}
\author[Packer]{Judith~A.~Packer}
\author[Raeburn]{Iain~Raeburn}
\address{Iain Raeburn, School  of Mathematics and Applied Statistics, University of Wollongong, NSW 2522, Australia}
\email{raeburn@uow.edu.au}
\author[Ramsay]{Arlan Ramsay}

\begin{abstract}
A multiresolution analysis for a Hilbert space realizes the Hilbert space as the direct limit of an increasing sequence of closed subspaces. In a previous paper, we showed how, conversely, direct limits could be used to construct Hilbert spaces which have multiresolution analyses with desired properties. In this paper, we use direct limits, and in particular the universal property which characterizes them, to construct wavelet bases in a variety of concrete Hilbert spaces of functions. Our results apply to the classical situation involving dilation matrices on $L^2(\R^n)$, the wavelets on fractals studied by Dutkay and Jorgensen, and Hilbert spaces of functions on solenoids.
\end{abstract}

\thanks{This research was supported by the Australian Research Council, the National Science Foundation (through grant DMS-0701913), the Research Council of Norway, and the University of 
Oslo.}

\maketitle

\section*{Introduction}

Suppose that $H$ is a Hilbert space equipped with a unitary operator $D$, which we think of as a dilation, and a unitary representation $T:\Gamma\to U(H)$ of an abelian group, which we think of as a group of translations. A multiresolution analysis (MRA) for $(H,D,T)$ consists of an increasing sequence of closed subspaces $V_n$, whose union is dense, whose intersection is $\{0\}$, and which satisfy $D(V_n)=V_{n+1}$, together with a \emph{scaling vector} $\phi\in V_0$ whose translates $T_\gamma\phi$ form an orthonormal basis for $V_0$; in a generalized multiresolution analysis (GMRA), the existence of the scaling vector is relaxed to the requirement that $V_0$ is $T$-invariant. MRAs and GMRAs play an important role in the construction of wavelets: a wavelet is a vector $\psi$ whose translates form an orthonormal basis for $W_0:=V_1\ominus V_0$, and then $\{D^jT_\gamma\psi:j\in \Z,\;\gamma\in\Gamma\}$ is an orthonormal basis for $H$. A famous theorem of Mallat \cite{mal} gives a procedure for constructing wavelets in the Hilbert space $L^2(\R)$, starting from a quadrature mirror filter, which is a function $m:\T\to \C$ satisfying $|m(z)|^2+|m(-z)|^2=2$, and proceeding through an MRA for the usual dilation operator and integer translations. Baggett, Courter, Merrill, Packer and Jorgensen have generalized Mallat's construction to GMRAs \cite{bcm, bjmp}.

Writing a Hilbert space $H$ as an increasing union of closed subspaces $V_n$ amounts to realizing $H$ as a Hilbert-space direct limit $\varinjlim V_n$. In \cite{gpots}, Larsen and Raeburn constructed MRAs for $L^2(\R)$ by constructing a direct system based on a single isometry $S_m$ on $L^2(\T)$ associated to a quadrature mirror filter $m$, and using the universal property of the direct-limit construction to identify the direct limit $\varinjlim (L^2(\T),S_m)$ with $L^2(\R)$. This yielded a new proof of Mallat's theorem. Subsequently the present authors used a similar construction to settle a question about multiplicity functions of generalized multiresolution analyses \cite{ijkln}. 

Here we will show that the universal properties of direct limits provide useful insight in a variety of situations involving wavelets and their generalizations. Our techniques provide efficient proofs of known results concerning classical wavelets and the wavelets on fractals studied by Dutkay and Jorgensen \cite{dutjor42}. We also obtain some interesting new results. We provide, building on our previous work in \cite{ijkln}, easily verified and very general criteria which imply that the isometries $S_m$ associated to filters are pure isometries (see Theorem~\ref{larryspurecond}). We use our direct-limit approach, and in particular the uniqueness of such limits, to settle a question of Ionescu and Muhly \cite{im} about the support of measures in realizations of MRAs in $L^2$-spaces on solenoids. 

We begin with a short section in which we recall general results on direct limits and MRAs from \cite{ijkln}, and indicate what extra information is needed to yield wavelet bases associated with these MRAs. In an attempt to emphasize how general our approach is, we will work whenever possible with an abstract translation group $\Gamma$, and for most purposes this poses no extra difficulty. In \S\ref{filisom}, we discuss the filters from which we build MRAs and the filter banks from which we build wavelet bases.  One key hypothesis in our general theory says that the isometry $S_m$ associated to a filter is a pure isometry, in the sense that its Wold decomposition has no unitary summand, and we prove our new criterion for pureness in \S\ref{sec:larry}.

In \S\ref{sec:id} we prove our main theorem on identifying direct limits, and illustrate its usefulness by applying it in the classical situation of a low-pass filter associated to dilation by an expansive integer matrix on $\R^n$. In the next two sections, we give several other applications of this theorem. The first involves the wavelets on fractals studied by Dutkay and Jorgensen. Starting with a filter which is definitely not low-pass, we run our direct-limit construction, and identify the direct limit as a Hilbert space of functions on a ``filled-in Cantor set'' constructed in \cite{dutjor42}. Second, under a nonsingularity hypothesis on the filter $m$, we realize our direct limits as spaces of functions on solenoids. This realization applies to both the classical case and the fractal case, and in both cases comparing the solenoidal realization with the original gives interesting information: in the fractal  case, we recover Dutkay's Fourier transform from \cite{dut}, and in the classical case, we deduce that the measure defining the $L^2$-space on the solenoid is supported on a ``winding line,'' thereby confirming a conjecture of Ionescu and Muhly \cite{im}. In the final section, we show that our methods can be used to obtain (a slight variation of) a theorem of Jorgensen on wavelet representations of the Baumslag-Solitar group~\cite{Jor}.

\subsection*{Notation and standing assumptions} We consider an additive countable abelian group $\Gamma$ and its compact dual group $\widehat\Gamma$. We write 
$\int_{\widehat\Gamma}f(k)\,dk$ for the integral of $f$ with respect to normalized Haar measure on $\widehat\Gamma$.  

Throughout the paper, we consider an injective endomorphism $\alpha$ of $\Gamma$ such that $\alpha(\Gamma)$ has finite index $N$ in $\Gamma$ and $\bigcap_{n\geq 0}\alpha^n(\Gamma)=\{0\}$.  We write $\alpha^*$ for the endomorphism $\omega\mapsto \omega\circ\alpha$ of $\widehat\Gamma$; observe that $\alpha^*$ is surjective, that $|\ker\alpha^*|=N$, and that $\bigcup_{n\geq 0}\ker\alpha^{*n}$ is dense in $\widehat\Gamma$. The example to bear in mind is the endomorphism of $\Gamma=\Z$ defined by $\alpha(n)=Nn$, when $\alpha^*$ is the endomorphism $z\mapsto z^N$ of $\T$. To simplify formulas, we sometimes write $(K,\beta)$ for $(\widehat\Gamma,\alpha^*)$.

\section{Wavelet bases in direct limits}\label{abstractbs}

Suppose that $S$ is an isometry on a Hilbert space $H$, and let $(H_\infty,U_n)$ be the Hilbert-space direct limit of the direct system $(H_n,T_n)$ in which each $(H_n,T_n)=(H,S)$. We proved in \cite[Theorem~5]{ijkln} that there is a unitary operator $S_\infty$ on $H_\infty$ characterized by $S_{\infty}U_n=U_nS=U_{n-1}$,
and that the subspaces $V_n$ of $H_\infty$ defined by
\begin{equation}\label{defVn}
V_n:=\begin{cases}
U_n(H)&\text{ if $n\geq 0$}\\
S_\infty^{|n|}(V_0)&\text{ if $n< 0$}
\end{cases}
\end{equation}
satisfy $V_n\subset V_{n+1}$, $\overline{\bigcup_{n\in \Z}V_n}=H_\infty$ and $S_{\infty}(V_{n+1})=V_n$. In addition, we have $\bigcap_{n\in \Z} V_n=\{0\}$ if and only if $S$ is a pure isometry, in which case the subspaces $W_n:=V_{n+1}\ominus V_n$ give an orthogonal decomposition $H_\infty=\bigoplus_{n\in\Z} W_n$. 

Now suppose that $\mu:\Gamma\to U(H)$ is a unitary representation such that $S\mu_\gamma=\mu_{\alpha(\gamma)}S $ for $\gamma\in \Gamma$. Then we proved in \cite[Theorem~5(d)]{ijkln} that there is a representation $\mu_\infty$ of $\Gamma$ on $H_\infty$ characterized by  $\mu_\infty(\gamma)U_n=U_n\mu_{\alpha^n(\gamma)}$; we then have $S_\infty\mu_\infty(\gamma)=\mu_\infty(\alpha(\gamma))S_\infty$, and the triple $(\{V_n\},\mu_\infty,S_\infty^{-1})$ is a generalized multiresolution analysis (GMRA) for $H_\infty$ if and only if $S$ is a pure isometry. 

At this point, we ask what extra input we need to ensure that this GMRA is associated to a wavelet or multiwavelet basis for $H_\infty$.

\begin{prop}\label{genwavelet}
Suppose that $S$ is a pure isometry on $H$. Suppose there are a Hilbert space 
$L$, a unitary representation $\rho:\Gamma\to U(L)$, an orthonormal set $B$ in $L$ such that $\{\rho_\gamma l:l\in B,\gamma\in \Gamma\}$ is an orthonormal basis for $L$, and a unitary isomorphism $S_1$ of $L$ onto $(SH)^\perp$ such that $S_1\rho_\gamma=\mu_{\alpha(\gamma)}S_1$. Then
\begin{equation}\label{genwaveletbasis}
\{S_\infty^{-j}\mu_\infty(\gamma)\psi:j\in \Z, \gamma\in \Gamma, \psi\in U_1S_1(B)\}
\end{equation}
is an orthonormal basis for $H_\infty$.
\end{prop}

\begin{proof}
We know that $U_1$ is an isomorphism of $H$ onto $V_1$, and $U_1(SH)=U_0H=V_0$, so $U_1$ is an isomorphism of $(SH)^\perp$ onto $W_0:=V_1\ominus V_0$. Thus $\{U_1S_1\rho_\gamma l:l\in B\}$ is an orthonormal basis for $W_0$. Now $S_{\infty}^{-j}$ maps $W_0$ onto $W_j$, and hence
\begin{equation}\label{foundon}
\{S_\infty^{-j}U_1S_1\rho_\gamma l:j\in \Z, \gamma\in \Gamma, l\in B\}
\end{equation}
is an orthonormal basis for $H_\infty$. But
\[
U_1S_1\rho_\gamma=U_1\mu(\alpha(\gamma))S_1=\mu_\infty(\gamma)U_1S_1,
\]
so \eqref{foundon} is the desired orthonormal basis \eqref{genwaveletbasis}.
\end{proof}

\section{Filters and isometries}\label{filisom}

In this section we will only use the dual endomorphism $\alpha^*$, so we simplify notation by writing $(K,\beta)$ for $(\widehat\Gamma,\alpha^*)$. Recall that $\beta$ is surjective and $N:=|\ker\beta|$ is finite.

A \emph{filter for $\beta$} is a Borel function $m:K\to \C$ such that
\begin{equation}\label{deffilter}
\sum_{a\in\ker\beta} |m(ak)|^2=N\ \text{ for almost all $k\in K$.}
\end{equation}
A \emph{filter bank for $\beta$} consists of Borel functions $m_a:K\to \C$ parametrized by $a\in \ker \beta$ such that
\begin{equation}\label{deffilterbank}
\sum_{d\in \ker\beta} m_a(dk)\overline{m_b(dk)}=\delta_{a,b}N\ \text{ for almost all $k\in K$;}
\end{equation}
Equation~\eqref{deffilterbank} says that the matrix $\big(N^{-1/2}m_a(dk)\big)_{a,d}$ is unitary for almost all $k$; in particular, each $m_a$ is a filter in its own right.

\begin{examples}\label{exsoffilters} (a) In the classical situation, we have $\Gamma=\Z$, $K=\T$, $\beta(z)=z^2$ and $N=2$, and in this case we recover the usual notions of conjugate mirror filter and filter bank with perfect reconstruction. More generally, we could take for $\beta$ the endomorphism of $\T^n$ induced by an integer matrix $B$: $\beta(e^{2\pi ix})=e^{2\pi i Bx}$ for $x\in \R^n$, in which case $N=|\det B|$. 

(b) To get a filter for a more general $\beta\in \End K$, choose characters $\gamma_0,\dots, \gamma_{N-1}$ in $\widehat K$ such that $(\ker \beta)^{\wedge}=\{\gamma_j|_{\ker\beta}: 0\leq j\leq N-1\}$. Then for every unit vector $c=(c_j)$ in $\C^N$,  $m(k):=\sum_{j=0}^{N-1}N^{1/2}c_j\gamma_j(k)$ defines a filter $m$ for $\beta$. To see this we just need to recall that the characters form an orthonormal basis for $\ell^2((\ker \beta)^{\wedge})$, and compute:
\begin{align*}
\sum_{a\in\ker\beta} |m(ak)|^2
&=\sum_{a\in\ker\beta} \sum_{i,j=1}^{N-1} Nc_i\gamma_i(ak)\overline{c_j
\gamma_j(ak)}\\
&=\sum_{i,j=0}^{N-1} Nc_i\gamma_i(k)\overline{c_j\gamma_j(k)}\Big(\sum_{a\in\ker\beta}\gamma_i(a)\overline{\gamma_j(a)}\Big)\\
&=\sum_{j=0}^{N-1} N|c_j|^2|\gamma_j(k)|^2,
\end{align*}
which is $N$ because $\gamma_j(k)\in \T$ and $c$ is a unit vector.

(c) To construct filter banks, we generalize a method from \cite{gopbur}. Choose an orthonormal basis $c_a=(c_{a,j})$ for $\C^{N}$, and take $m_a(k)=\sum_{j=0}^{N-1}N^{1/2}c_{a,j}\gamma_j(k)$. Then, as in the previous calculation,
\[
\sum_{d\in \ker\beta} m_a(dk)\overline{m_b(dk)}
=\sum_{i,j=0}^{N-1} Nc_{a,i}\gamma_i(k)\overline{c_{b,j}\gamma_j(k)}
\Big(\sum_{a\in\ker\beta}\gamma_i(a)\overline{\gamma_j(a)}\Big)=N(c_a\,|\,c_b).
\]
\end{examples}

The next lemma is well-known in special cases (see \cite{bj}, for example).

\begin{prop}\label{SmaCuntz} \textnormal{(a)} If $m$ is a filter for $\beta$, then the formula $(S_mf)(k)=m(k)f(\beta(k))$ defines an isometry $S_m$ on $L^2(K)$. 

\textnormal{(b)} If $\{m_a:a\in \ker \beta\}$ is a filter bank for $\beta$, then $\{S_{m_a}:a\in \ker\beta\}$ satisfies the Cuntz relation
\[
\sum_{a\in \ker\beta}S_{m_a}S_{m_a}^*=1.
\]
\end{prop}

Part (a) implies that for every filter $m$ we can run the argument of \S\ref{abstractbs} with $S=S_m$; if $S_m$ is pure, we obtain a GMRA for the direct limit $L^2(K)_\infty$. Part (b) implies that for every $a$,  $S_1:=\bigoplus_{b\in\ker\beta,\;b\not=a} S_{m_b}$ is an isometry of $\bigoplus_{b\not=a}L^2(K)$ onto 
\[
(S_{m_a}(L^2(K)))^\perp =(S_{m_a}S_{m_a}^*(L^2(K)))^\perp=\bigoplus_{b\in\ker\beta,\;b\not=a} S_{m_b}S_{m_b}^*(L^2(K));
\]
thus, when a filter $m$ is a member of a filter bank, we can use Proposition~\ref{genwavelet} to generate a multiwavelet basis for $L^2(K)_\infty$. 

To prove Proposition~\ref{SmaCuntz}, we need an elementary lemma. Notice that our countability hypothesis on $\Gamma=\widehat K$ implies that there is always a Borel section $c$ for the surjection $\beta:K\to K$.

\begin{lemma}\label{changevbles}
Suppose that $c:K\to K$ is a Borel map such that $\beta(c(k))=k$ for all $k\in K$. Then for every continuous function $f$ on $K$ we have
\begin{itemize}
\item[(a)] $\int_K f(\beta(k))\,dk=\int_K f(k)\,dk$, and
\smallskip
\item[(b)] $\int_K f(k)\,dk=\int_K N^{-1}\big(\textstyle{\sum_{a\in \ker\beta} f(ac(k))}\big)\,dk$.
\end{itemize}
\end{lemma} 

\begin{proof} For (a), we define $I(f):=\int_K f(\beta(k))\,dk$. Since $\beta$ is surjective, it follows easily from the translation invariance of Haar measure on $K$ that $I$ is also a translation-invariant integral on $K$; since $I(1)=1$, it must be the Haar integral, and (a) follows.

For (b), we use (a) to simplify the right-hand side:
\begin{align*}
\int_K N^{-1}\big(\textstyle{\sum_{a\in \ker\beta} f(ac(k))}\big)\,dk
&=\sum_{a\in \ker\beta}\int_K N^{-1}f(ac(k))\,dk\\
&=\sum_{a\in \ker\beta}\int_K N^{-1}f(\beta(ac(k)))\,dk\\
&=\sum_{a\in \ker\beta}\int_K N^{-1}f(k)\,dk,
\end{align*}
which since $N=|\ker\beta|$ gives (b).
\end{proof}

\begin{proof}[Proof of Proposition~\ref{SmaCuntz}]
To see that $S_m$ is an isometry, we compute using part (b) of Lemma~\ref{changevbles}:
\begin{align}
\label{Smisometric}\|S_mf\|^2&=\int_K |m(k)f(\beta(k))|^2\,dk\\
&=\int_K N^{-1}\big(\textstyle{\sum_{a\in \ker\beta} |m(ac(k))f(\beta(ac(k)))|^2}\big)\,dk\\
\notag&=\int_K N^{-1}\big(\textstyle{\sum_{a\in \ker\beta} |m(ac(k))|^2}\big)|f(k)|^2\,dk,
\end{align}
which by the filter equation \eqref{deffilter} is precisely $\|f\|^2$.

For (b), we use Lemma~\ref{changevbles}(b) again to check that
\[
(S_{m_a}^*f)(k)=N^{-1}\sum_{d\in \ker\beta}\overline{m_a(dc(k))}f(dc(k))=N^{-1}\sum_{\beta(l)=k}\overline{m_a(l)}f(l),
\]
compute
\[
(S_{m_a}S_{m_a}^*f)(k)=m_a(k)N^{-1}\sum_{\beta(l)=\beta(k)}\overline{m_a(l)}f(l)=
m_a(k)N^{-1}\sum_{d\in\ker\beta}\overline{m_a(dk)}f(dk),
\]
and add to get
\[
\sum_{a\in \ker\beta}(S_{m_a}S_{m_a}^*f)(k)=N^{-1}\sum_{d\in\ker\beta}\Big(\sum_{a\in \ker\beta} m_a(k)\overline{m_a(dk)}\Big)f(dk).
\]
Now the term in brackets is the inner product of two columns of the unitary matrix $(m_a(dk))_{a,d}$, and hence vanishes unless $d=1$, in which case we are left with $N^{-1}Nf(k)$.  
\end{proof}

\section{When $S_m$ is a pure isometry}\label{sec:larry}

A crucial hypothesis in the general theory of \S\ref{abstractbs} is that the isometry $S$ is pure. Our next theorem gives easily verifiable criteria which imply that an isometry of the form $S_m$ is pure. We stress that this is not an elementary fact: the proof uses results from \cite{ijkln} which rely on the reverse martingale convergence theorem.

\begin{theorem}\label{larryspurecond}
Suppose that $B$ is a Borel subset of $\widehat \Gamma$ and $m:\widehat\Gamma\to \C$ is a Borel function such that
\begin{equation}\label{genfiltereq}
\sum_{\alpha^*(\zeta)=\omega}|m(\zeta)|^2=N\chi_B(\omega)\ \text{ for almost all $\omega\in\widehat\Gamma$, }
\end{equation}
and define $S_m:L^2(B)\to L^2(B)$ by $(S_mf)(\omega)=m(\omega)f(\alpha^*(\omega))$. If either 
\begin{enumerate}
\item[\textnormal{(a)}] $\widehat\Gamma\backslash B$ has positive Haar measure, or 
\smallskip
\item[\textnormal{(b)}]
$|m(\omega)|\not=1$ on a set of positive measure,
\end{enumerate} 
then $S_m$ is a pure isometry.
\end{theorem}

\begin{proof}
In the language of \cite{ijkln}, the hypothesis on $m$ says that ``$m$ is a filter relative to the multiplicity function $\chi_B:\widehat \Gamma\to \{0,1\}$ and the endomorphism $\beta:=\alpha^*$." We are not assuming that $m$ is a low-pass filter, but that hypothesis is not used in the proof of \cite[Theorem~8]{ijkln} until after Proposition~12. So we know from \cite[\S4]{ijkln} that $S_m$ is an isometry. We will assume that $S_m$ is not pure, and aim to prove that neither (a) nor (b) holds. Saying that $S_m$ is not pure means that $R_\infty:=\bigcap_{n=0}^\infty S_m^nL^2(B)$ is non-zero, and hence that there exists a unit vector $f$ in $R_\infty$. Proposition~12 of \cite{ijkln} implies that the functions $f_n:=S_m^{*n}f$ satisfy
\begin{equation}\label{fromprop12}
f_n(\beta^n(\omega))\to 1\ \text{ as $n\to \infty$ for almost all $\omega\in\widehat\Gamma$.}
\end{equation}
We claim that  $|m(\omega)|\geq 1$ for almost all $\omega$.

To establish this claim, we again suppose not, so that there exists $\epsilon>0$ and a Borel set $C$ of positive (Haar) measure such that $|m(\omega)|\leq 1-\epsilon$ for $\omega\in C$. Let $\delta>0$. Then we can deduce from \eqref{fromprop12} and Egorov's theorem that there exist a Borel set $E\subset C$ and $M\in \N$ such that $E$ has positive measure and
\begin{equation*}
n\geq M\text{ and }\omega\in E\Longrightarrow 1-\delta<|f_n(\beta^n(\omega))|<1+\delta.
\end{equation*}
Lemma~\ref{changevbles} implies that $\beta$ is measure-preserving, so the Poincar{\' e} recurrence theorem (as in \cite[Theorem~2.3.2]{pet}) implies that there is a Borel set $E'\subset E$ such that $E\setminus E'$ has measure zero and 
$\{n\in \N:\beta^n(\omega)\in E'\}$ is infinite for every $\omega\in E'$. Writing $E'=\bigcup_{n=M}^\infty\{\omega\in E':\beta^n(\omega)\in E'\}$ implies that there exists $n\geq M$ such that $F:=\{\omega\in E':\beta^n(\omega)\in E'\}$ has positive measure. In particular, for $\omega\in F$, $\beta^n(\omega)$ belongs to $C$, and 
\begin{align*}
1-\delta&\leq |f_n(\beta^n(\omega))|=|(S_mf_{n+1})(\beta^n(\omega))|\\
&=
|m(\beta^n(\omega))f_{n+1}(\beta^{(n+1)}(\omega))|\\
&\leq (1-\epsilon)(1+\delta).
\end{align*}
Since this is true for every $\delta>0$, we can let $\delta\to 0+$ and deduce that $1\leq 1-\epsilon$, which is a contradiction. 

Thus $|m(\omega)|\geq 1$ for almost all $\omega$, and the left-hand side of the filter equation \eqref{genfiltereq} is $\geq N$ for almost all $\omega$. Since the right-hand side of is $\leq N$, both sides must equal $N$, which implies that $\chi_B(\omega)=1$ and $|m(\omega)|=1$ for almost all $\omega$, so that neither (a) nor (b) holds, as required.
\end{proof}

\begin{remark}
When $B=\widehat\Gamma =\T$, this follows from Theorem~3.1 of \cite{bj}. That theorem also asserts that when $|m|\equiv 1$, the space $R_\infty$ is spanned by a single function $\xi:\T\to\T$, and that $m$ then has the form $m(z)=\lambda\xi(z)\overline{\xi(z^N)}$ for some $\lambda\in \T$. These extra assertions also extend to the general case. 

To see this, we again consider a unit vector $f$ in $R_\infty$, and deduce from the equations $f=S_m^nf_n$ and $|m|\equiv 1$ that
\[
|f(\omega)|=\Big|\prod_{k=0}^{n-1} m(\beta^{k}(\omega))f_n(\beta^{n}(\omega))\Big|=|f_n(\beta^{n}(\omega))|.
\]
Thus $|f(\omega)|=|f(\omega\zeta)|$ for almost all $\omega$ and every $\zeta\in \ker\beta^{n}$. Since the right-regular representation $\rho$ is continuous and $\bigcup_{n\geq 1}\ker\beta^{n}$ is dense in $\widehat \Gamma$, this implies that $\rho_\zeta(|f|)=|f|$ for all $\zeta\in\widehat \Gamma$. The Fourier transform $|f|^\wedge$ then satisfies $\zeta(\gamma)|f|^\wedge(\gamma)=|f|^\wedge(\gamma)$ for all $\zeta\in \widehat\Gamma$ and all $\gamma\in \Gamma$, so $|f|^\wedge(\gamma)=0$ for $\gamma\not=0$, and $|f|$ is constant. 

So $|f|$ is constant for every $f\in R_\infty$. This implies that $R_\infty$ is one-dimensional: if $f,g\in R_\infty$ are non-zero, then $2\Real f\overline g=|f+g|^2-|f|^2-|g|^2$ and $2\Imag f\overline g=|f+ig|^2-|f|^2-|g|^2$ are constant, so $f\overline g$ is constant and $f=(f\overline{g})g/|g|^2$ is a constant multiple of $g$. If we choose a spanning element $\xi$ which is a unit vector, so that $|\xi|\equiv 1$, then $S_m\xi$ is also a unit vector in $R_\infty$. Thus there exists $\lambda\in \T$ such that $S_m\xi= \lambda\xi$, which says that $m(\omega)\xi(\beta(\omega))=\lambda\xi(\omega)$ for almost all $\omega$.
\end{remark}

\section{Identifying the direct limit}\label{sec:id}

The universal property of the direct limit implies that, to identify $H_\infty$ with a given space $K$, we only need to find isometries $R_n:H\to K$ such that $R_{n+1}S=R_n$ and $\bigcup_{n=0}^\infty R_nH$ is dense in $K$. In \cite{gpots}, for example, we applied this strategy to identify $L^2(\T)_\infty$ with $L^2(\R)$ when $S$ is the isometry $S_m$ associated to a quadrature mirror filter on $\T$. If we have a candidate for the unitary $S_\infty$, it is even easier.

\begin{theorem}\label{iddirlim}
Suppose that $\mu:\Gamma\to U(H)$ is a unitary representation, and $S$ is an isometry on $H$ such that $S\mu_\gamma=\mu_{\alpha(\gamma)}S$ for $\gamma \in \Gamma$. Suppose that $\lambda:\Gamma\to U(K)$ is a unitary representation and $D$ is a unitary operator on $K$ such that $D\lambda_\gamma D^*=\lambda_{\alpha(\gamma)}$ for $\gamma\in \Gamma$. If there is an isometry
$R:H\to K$ such that
\[
\textnormal{(a)}\ RS=DR, \ \text{ and }\ \textnormal{(b)}\ R\mu_\gamma=\lambda_\gamma R\text{ for $\gamma\in \Gamma$,}
\]
then there is an isomorphism $R_\infty$ of $H_\infty$ onto the subspace $\overline{\bigcup_{n=0}^\infty D^{-n}R(H)}$ of $K$ such that $R_\infty S_\infty R_\infty^*=D$ and $R_\infty \mu_\infty R_\infty^*=\lambda$. The subspaces $D^{-n}R(H)$ form a GMRA of $R_\infty(H_\infty)$ relative to $D$ and $\lambda$ if and only if $S$ is a pure isometry.
\end{theorem}

\begin{proof}
We define $R_n:H\to K$ by $R_n=D^{-n}R$. Then each $R_n$ is an isometry, and from (a) we have
\[
R_{n+1}S=D^{-(n+1)}RS=(D^{-n}D^{-1})(DR)=D^{-n}R=R_n.
\]
Thus the $R_n$ induce an isometry $R_\infty$ of $H_\infty$ into $K$, and this is a unitary isomorphism onto the subspace $\overline{\bigcup_{n=0}^\infty D^{-n}R(H)}$ of $K$. For each $n\geq 1$ we have
\[
R_\infty S_\infty U_n=R_\infty U_nS=R_nS=R_{n-1}=DD^{-n}R=DR_n=DR_\infty U_n,
\]
so $R_\infty$ intertwines $S_\infty$ and $D$. For $\gamma\in\Gamma$ and $n\geq 0$, we have
\begin{align*}
R_\infty \mu_\infty(\gamma) U_n&=R_\infty U_n\mu_{\alpha^n(\gamma)}=R_n\mu_{\alpha^n(\gamma)}\\
&=D^{-n}R\mu_{\alpha^n(\gamma)}=D^{-n}\lambda_{\alpha^n(\gamma)}R\\
&=\lambda_\gamma D^{-n}R=\lambda_\gamma R_n=\lambda_\gamma R_\infty U_n,
\end{align*}
and this implies that $R_\infty \mu_\infty(\gamma)R_\infty^*=\lambda_\gamma$. The last assertion holds because the subspaces $V_n$ defined by \eqref{defVn} are a GMRA for $\H_\infty$ if and only if $S$ is pure.
\end{proof}

To construct the isometry $R$ when $S$ is the isometry $S_m$ associated to a filter $m$, we use a scaling function $\phi$ for the filter. We illustrate how this works by applying Theorem~\ref{iddirlim} in the classical situation of a dilation by an integer matrix on $\R^n$, thereby showing that the approach taken in \cite{gpots} also covers this situation.

\begin{example}[Classical wavelets]
\label{dilatmatrix}
Let $A\in GL_n(\Z)$ be an integer matrix  such that every eigenvalue $\lambda$ has $|\lambda|>1$, and define $\alpha\in \End\Z^n$ by $\alpha(k)=Ak$ (using multi-index notation). Note that $N:=|\Z^n/A\Z^n|=|\det A|$. The dual endomorphism $\alpha^*$ of $\T^n$ is given on $e^{2\pi ix}:=(e^{2\pi ix_1},\dots,e^{2\pi ix_n})$ by $\alpha^*(e^{2\pi ix})=e^{2\pi iA^tx}$. Suppose that $m:\T^n\to \C$ is a filter which is \emph{low-pass}, in the sense that $m(1)=N^{1/2}$, and is Lipschitz near $1$; suppose also that $m$ is non-vanishing on a suitably large neighbourhood of $1$ (this is Cohen's condition; see \cite[Theorem~1.9]{str}, for example). Theorem~\ref{larryspurecond} implies that $S_m$ is a pure isometry. 

Under our hypotheses on $m$ the infinite product\footnote{The assertions in this sentence are all well-known (see \cite{str}, for example), but it is hard to point to an efficient derivation. They can, however, be deduced from the more general results in \cite[Proposition~3.1]{bjmp} and \cite[Lemma~3.3]{bcm}; there we need to take the multiplicity function to be identically $1$ on $\T^n$, so that the matrix $H$ consists of the single function denoted here by $m$, and observe that in this case the functions $\tilde M^n$ and $M^n$ in \cite[\S3]{bcm} coincide.}
\begin{equation}\label{scalingdef}
\phi(x)=\textstyle{\prod_{n=1}^\infty N^{-1/2}m(e^{2\pi i(A^t)^{-n}x})}
\end{equation}
converges pointwise almost everywhere for $x\in \R^n$ and in $L^2(\R^n)$ to a unit vector $\phi\in L^2(\R^n)$; the limit $\phi$ is continuous near $0$, satisfies $\phi(0)=1$, \begin{gather}
\label{scaling1}N^{1/2}\phi(A^tx)=m(e^{2\pi ix})\phi(x),\ \mbox{ and}\\
\label{scaling2}\sum_{k\in \Z^n}|\phi(x+k)|^2=1
\end{gather}
for almost all $x\in \R^n$.

We now define $R:L^2(\T^n)\to L^2(\R^n)$ by 
\[
(Rf)(x)=f(e^{2\pi ix})\phi(x).
\]
With $B=\prod_{j=1}^n[0,1)$, $\R^n$ is the disjoint union of the sets $B+k$ for $k\in \Z^n$, and
\begin{align*}
\|Rf\|^2&=\sum_{k\in \Z^n}\int_{B}|f(e^{2\pi ix})\phi(x+k)|^2\,dx\\
&=\int_{B}|f(e^{2\pi ix})|^2\big(\textstyle{\sum_{k\in \Z^n}}|\phi(x+k)|^2\big)\,dx\\
&=\|f\|^2
\end{align*} 
by \eqref{scaling2}. Thus $R$ is an isometry. With $(Dg)(x):=N^{1/2}g(A^tx)$, the scaling equation \eqref{scaling1} gives
\[
(RS_mf)(x)=m(e^{2\pi ix})f(e^{2\pi iA^tx})\phi(x)=N^{1/2}f(e^{2\pi iA^tx})\phi(A^tx)=(DRf)(x),
\]
and with $\mu:\Z^n\to U(L^2(\T^n))$ defined by $(\mu_kf)(z)=z^kf(z)$ and $\lambda:\Z^n\to U(L^2(\R^n))$ by $(\lambda_kf)(x)=e^{2\pi i x\cdot k}g(x)$, we can easily check that $R\mu_k=\lambda_kR$. Thus Theorem~\ref{iddirlim} implies that there is an isomorphism $R_\infty$ of $L^2(\T^n)_\infty$ onto the subspace $\overline{\bigcup_{j=0}^\infty D^{-j}R(L^2(\T^n))}$ of $L^2(\R^n)$ which intertwines $(S_\infty,\mu_\infty$) and $(D,\lambda)$. Since $R$ is an isometry, the functions $e_k\phi:x\to e^{2\pi ik\cdot x}\phi(x)$ form an orthonormal basis for $V_0:=R(L^2(\T^n))$, and hence the functions $D^{-j}(e_k\phi)$ form an orthonormal basis for $V_j:=D^{-j}R(L^2(\T^n))$. Thus we can run the standard argument (as on page 212 of \cite{bcm}, for example) to see that $\bigcup V_j$ is dense in $L^2(\R^n)$. We deduce that the subspaces $\{V_j\}$ form a multiresolution analysis for $L^2(\R^n)$.

Now suppose that $m_1:=m$ is part of a filter bank $\{m_w:w\in \ker\alpha^*\}$ parametrized by
\[
\ker\alpha^*=\{w\in\T^n:w=e^{2\pi ix} \text{ for some $x\in \R^n$ such that $A^tx\in \Z^n$}\}.
\]
(It is known that for every filter $m$ there is always a filter bank containing $m$ \cite[page~494]{bow}, but our construction depends on fixing one.) Since $\{S_{m_w}:w\in \ker\alpha^*\}$ is a Cuntz family,
\begin{equation}\label{defS1}
\textstyle{S_1:=\bigoplus_{w\not= 1}S_{m_w}:\bigoplus_{w\not= 1}L^2(\T^n)\to L^2(\T^n)}
\end{equation}
is an isometry with range $(S_mL^2(\T^n))^\perp$. Thus we can apply Proposition~\ref{genwavelet} with $S_1$ given by \eqref{defS1}.
Note that $D^{-1}R$ is an isomorphism of $(S_mL^2(\T^n))^\perp$ onto $W_0:=V_1\ominus V_0$. Let $1_w$ denote the constant function $1$ in the $w$th copy of $L^2(\T^n)$, so that the functions $\{x\mapsto e^{2\pi ik\cdot x}1_w:w\in\ker\alpha^*,\ w\not= 1\}$ form an orthonormal basis for $\bigoplus_{w\not= 1}L^2(\T^n)$, and set
\[
\psi_w(x):=D^{-1}RS_11_w(x)=N^{-1/2}m_w(e^{2\pi i(A^t)^{-1}x})\phi((A^t)^{-1}x).
\]
Proposition~\ref{genwavelet} implies that the functions
\[
\psi_{w,j,k}(x):=N^{j/2}e^{2\pi ik\cdot (A^t)^{j}x}\psi_w((A^t)^jx)
\]
form an orthonormal basis for $L^2(\R^n)$, and the inverse Fourier transforms $\{\check\psi_w:w\in\ker\alpha^*,\ w\not=1\}$ form a multi-wavelet for $L^2(\R^n)$.
\end{example}

\begin{example}
Consider the multiplicity function $\chi_B:\T\to \{0,1\}$ associated to the interval $(-\frac{1}{3}, \frac{1}{3}]$ (or rather to the set $B:=\{e^{2\pi ix}:x\in (-\frac{1}{3}, \frac{1}{3}]\}$). We can check that the function $m:e^{2\pi ix}\mapsto 2^{1/2}\chi_{(-\frac{1}{6}, \frac{1}{6}]}(x)$ satisfies the generalized filter equation \eqref{genfiltereq} with $N=2$, and hence Theorem~\ref{larryspurecond} implies that $S_m:L^2(B)\to L^2(B)$ is a pure isometry. The function $\phi:=\chi_{(-\frac{1}{3}, \frac{1}{3}]}$ satisfies the scaling equation $2^{1/2}\phi(2x)=m(e^{2\pi ix})\phi(x)$, so in parallel with the classical case we define $R:L^2(B)\to L^2(\R)$ by
\[
(Rf)(x)=f(e^{2\pi ix})\chi_{(-\frac{1}{3}, \frac{1}{3}]}(x).
\]
Calculations show that the usual dilation operator defined by $(D\xi)(x)=2^{1/2}\xi(2x)$ satisfies $DR=RS_m$, and that $R$ intertwines the representations $\mu$ and $\lambda$ of $\Z$ defined by $(\mu_nf)(z)=z^nf(z)$ and $(\lambda_n\xi)(x)=e^{2\pi inx}\xi(x)$. The range of $R$ is the subspace $L^2(-\frac{1}{3}, \frac{1}{3}]$ of $L^2(\R)$ consisting of functions which vanish for $|x|>\frac{1}{3}$, and $D^{-n}(L^2(-\frac{1}{3}, \frac{1}{3}])=L^2(-\frac{2^n}{3}, \frac{2^n}{3}]$, so the dominated convergence theorem implies that $\bigcup_{n=0}^\infty D^{-n}R(L^2(B))$ is dense in $L^2(\R)$. Thus Theorem~\ref{iddirlim} implies that the subspaces $D^{-j}R(L^2(B))$ form a GMRA for $L^2(\R)$. 

Since the functions $e_n:x\mapsto e^{2\pi inx}$ form an orthonormal basis for $L^2(-\frac{1}{2}, \frac{1}{2}]$, and since multiplication by $\phi=\chi_{(-\frac{1}{3}, \frac{1}{3}]}$ is the orthogonal projection on $L^2(-\frac{1}{3}, \frac{1}{3}]$, the functions $\lambda_n\phi$ form a Parseval frame for $RL^2(B)=L^2(-\frac{1}{3}, \frac{1}{3}]$. The inverse Fourier transform of $\lambda_n\phi$ is the translate $\check\phi(\cdot -n)$, and hence we have just shown that the inverse Fourier transforms $V_j:=(D^{-j}R(L^2(B)))^\vee$ form a frame multiresolution analysis in the sense of \cite{bl} --- indeed, we have just recovered Example 4.10(a) of \cite{bl}. 
\end{example}

\section{Wavelets associated to the Cantor set}\label{sec:cantor}

The characteristic function $\chi_C$ of the middle-third Cantor set in $[0,1]$ satisfies
\begin{equation}\label{selfsimofC}
\chi_C(3^{-1}x)=\chi_C(x)+\chi_C(x-2)\ \text{ for all $x\in \R$.}
\end{equation}
Dutkay and Jorgensen observed in \cite{dutjor42} that this is formally similar to saying that $\chi_C$ satisfies a scaling equation involving the dilation $(Df)(x)=f(3^{-1}x)$ and two translations. The right-hand side can be viewed as convolution with the measure $\delta_0+\delta_2$, which is the inverse Fourier transform of $1+z^2\in L^2(\T)$. So one is led to view $1+z^2$ as a filter, and consider the associated isometry on $L^2(\T)$.

We consider the function $m:\T\to \C$ defined by $m(z)=2^{-1/2}(1+z^2)$; the normalising factor of $2^{-1/2}$ ensures that $m$ satisfies
\begin{equation}\label{m_0cubicfilter}
|m(z)|^2+|m(\omega z)|^2+|m(\omega^2z)|^2=3,
\end{equation}
where $\omega:=e^{2\pi i/3}$ is a cube root of unity, so that $m$ is a filter for multiplication by $3$. Notice that $m$ is \emph{not} low-pass: it satisfies $m(1)=2^{1/2}$ rather than $m(1)=3^{1/2}$. A key point established in \cite{dutjor42} is that when we mimic the classical construction of wavelets on $\R$ using this filter, we wind up in a Hilbert space of functions determined by a measure which is supported on a set of Lebesgue measure $0$. Our goal in this section is to show that our recognition theorem also applies in this situation.

Theorem~\ref{larryspurecond} implies that the operator on $L^2(\T)$ defined by $(S_mf)(z)=m(z)f(z^3)$ is a pure isometry. With $\alpha\in\End\Z$ defined by $\alpha(n)=3n$ and  $\mu:\Z\to U(L^2(\T))$ given by $(\mu_nf)(z)=z^nf(z)$, we have $S_m\mu_n=\mu_{3n}S_m=\mu_{\alpha(n)}S_m$. We want to identify the direct limit $(L^2(\T)_\infty, S_\infty,\mu_\infty)$ using $\phi:=\chi_C$ as scaling function.

When we normalize $m$ by multiplying by $2^{-1/2}$, we need to multiply both sides of the scaling equation \eqref{selfsimofC} by $2^{-1/2}$, and hence the appropriate dilation operator is given by $(Df)(x)=2^{-1/2}f(3^{-1}x)$. Following \cite{dutjor42}, we define 
\[
\CR:=\bigcup\{ 3^{-n}(C+k):k,n\in \Z\},
\]
and let $\nu$ denote the Borel measure on $\CR$ which has $\nu(C)=1$, is invariant for the action of $\Z$ by translation on $\CR$, and satisfies 
\begin{equation}\label{dilatinvofnu}
\int f(x)\,d\nu(x)=2^{-1}\int f(3^{-1}x)\,d\nu(x)\ \text{ for every $f\in L^1(\CR,\nu)$.}
\end{equation}
(See \cite[Proposition~2.4]{dutjor42}.) Thus $D$ is a unitary operator on $L^2(\CR,\nu)$, and the scaling function $\chi_C$ is a unit vector. We define $\lambda: \Z\to U(L^2(\CR,\nu))$ by $(\lambda_nf)(x)=f(x-n)$. A straightforward calculation shows that $D\lambda_n=\lambda_{3n}D$, so that $D\lambda_nD^*=\lambda_{3n}$. 

\begin{prop}\label{idwithL2CR}
The direct limit $(L^2(\T)_\infty, S_\infty,\mu_\infty)$ is isomorphic to $(L^2(\CR,\nu),D,\lambda)$. The subspaces 
\[
V_n=\newspan\{D^{-n}\lambda_k(\chi_C):k\in \Z\}
\]
form an MRA for $L^2(\CR,\nu)$, and $\{\lambda_k(\chi_C):k\in \Z\}$ is an orthonormal basis for $V_0$.
\end{prop}

To apply Theorem~\ref{iddirlim}, we need an isometry $R:L^2(\T)\to L^2(\CR,\nu)$. This one looks a little different to those in the previous section because the scaling equation in the form \eqref{selfsimofC} involves a convolution rather than a pointwise multiplication in the Fourier domain.

\begin{lemma}\label{defR}
For $n\in \Z$, let $e_n$ denote the function $z\mapsto z^n$. Then there is an isometry $R$ of $L^2(\T)$ into $L^2(\CR,\nu)$ such that $Re_n=\lambda_n\chi_C=\chi_{C+n}$ for $n\in \Z$.
\end{lemma}

\begin{proof}
Since $\{e_n:n\in \Z\}$ is an orthonormal basis for $L^2(\T)$, it suffices for us to check that the elements $\lambda_n\chi_C=\chi_{C+n}$ form an orthonormal set in $L^2(\CR,\nu)$. Since singleton sets have $\nu$-measure zero, we can delete $1$ from $C$ without changing the element $\chi_C$ of $L^2(\CR,\nu)$; now the sets $C+n$ are disjoint, so the functions are mutually orthogonal, and since $\nu(C+n)=\nu(C)=1$, each $\chi_{C+n}$ is a unit vector.
\end{proof}

To get surjectivity of our isomorphism $R_\infty$, we need the following lemma\footnote{This result is stated as Proposition~2.8(iii) in \cite{dutjor42}, but there seems to be a gap in the proof. This was observed and fixed independently by Sam Webster and Kathy Merrill. The proof of Lemma~\ref{Samsfix} is similar to the proof in Sam's honours thesis (University of Newcastle, 2006); Kathy's argument is generalized in \cite{dmp}.}.

\begin{lemma}\label{Samsfix}
The functions 
\[
\{\chi_{3^{-n}(C+k)}=2^{-n/2}D^{-n}\lambda_k(\chi_C):n,k\in \Z\}
\]
span a dense subspace of $L^2(\CR,\nu)$.
\end{lemma}

\begin{proof}
Since $\CR=\bigcup_{n=0}^\infty 3^{-n}\big(\bigcup_{k\in \Z}(C+k)\big)$ is an increasing union of almost disjoint unions, two applications of the dominated convergence theorem show that it suffices to approximate functions $f$ with support in $3^{-N}(C+K)$ for fixed $N\geq 0$ and $K\in \Z$. Then $\lambda_{-K}D^Nf$ has support in $C$.

We now consider the sets $3^{-n}(C+k)$ which are contained in $C$. For each $n\geq 0$, there are exactly $2^n$ such sets, and they are disjoint; each 
\[
3^{-n}(C+k)=3^{-(n+1)}(C+3k)\cup 3^{-(n+1)}(C+3k+2).
\]
Thus two such sets are either disjoint or one is contained in the other, and 
\[
A:=\newspan\{\chi_{3^{-n}(C+k)}:n\geq 0,\ k\in \Z,\text{ and }3^{-n}(C+k)\subset C\}
\]
is a $*$-subalgebra of $C(C)$; since $A$ contains the characteristic functions of arbitrarily small sets, it separates points of $C$, and hence by the Stone-Weierstrass theorem is uniformly dense in $C(C)$. Since $\nu$ is inner regular and $C$ has finite measure, the restriction of $\nu$ to $C$ is a regular Borel measure, and $C(C)$ is dense in $L^2(C,\nu)$. Thus we can find a function $g$ in 
\[
\newspan\{\chi_{3^{-n}(C+k)}:n,k\in \Z\}=\newspan\{D^{-n}\lambda_k(\chi_C):n,k\in \Z\}
\]
such that $\|\lambda_{-K}D^Nf-g\|$ is small. Since $\lambda_K$ and $D^{-N}$ are unitary, $\|f-D^{-N}\lambda_Kg\|$ is also small. But
\[
D^{-N}\lambda_K(D^{-n}\lambda_k(\chi_C))=D^{-(N+n)}\lambda_{3^nK+k}(\chi_C),
\]
so $D^{-N}\lambda_Kg$ has the required form.
\end{proof}

\begin{proof}[Proof of Proposition~\ref{idwithL2CR}]
We next check that $RS_m=DR$ (equation (a) of Theorem~\ref{iddirlim}). For each $n\in \Z$, we have
\[
(DRe_n)(x)=(D\chi_{C+n})(x)=2^{-1/2}\chi_{C+n}(3^{-1}x)=2^{-1/2}\chi_{C}(3^{-1}(x-3n)),
\]
which in view of the scaling equation \eqref{selfsimofC} gives
\[
(DRe_n)(x)=2^{-1/2}\big(\chi_C(x-3n)+\chi_C(x-3n-2)\big)=R(2^{-1/2}(e_{3n}+e_{3n+2}))(x).
\]
Since 
\[
(S_me_n)(z)=2^{-1/2}(1+z^2)e_n(z^3)=2^{-1/2}(1+z^2)(z^{3n})=2^{-1/2}(e_{3n}+e_{3n+2})(z), 
\]
we deduce that $RS_m$ and $DR$ agree on the basis elements $e_n$, and hence are equal.

To check the hypothesis (b) of Theorem~\ref{iddirlim}, observe that $\mu_ne_k=e_{k+n}$. Thus for $n,k\in \Z$ we have
\[
(R\mu_n)e_k=Re_{n+k}=\chi_{C+n+k}=\lambda_n(\chi_{C+k})=(\lambda_nR)e_k.
\]
Now Theorem~\ref{iddirlim} gives an isometry $R_\infty$ of $(L^2(\T)_\infty, S_\infty,\mu_\infty)$ into $(L^2(\CR,\nu),D,\lambda)$. Since the range of $R$ contains the vectors $\lambda_n(\chi_C)$, it follows from Lemma~\ref{Samsfix} that $\bigcup_{n\geq 0} D^{-n}(R(L^2(\T)))$ is dense in $(L^2(\CR,\nu),D,\lambda)$, and the result follows. 
\end{proof}

To get a wavelet basis for $L^2(\CR,\nu)$, we observe that $m_0=m$ and $m_1(z)=z$, $m_2(z)=2^{-1/2}(1-z^2)$ form a filter bank: with $\omega=\exp(2\pi i/3)$, the matrix
\[
3^{-1/2}\begin{pmatrix}
m_0(z)&m_1(z)&m_2(z)\\
m_0(\omega z)&m_1(\omega z)&m_2(\omega z)\\
m_0(\omega^2 z)&m_1(\omega^2 z)&m_2(\omega^2 z)
\end{pmatrix}
\]
is unitary for every $z\in \T$. Proposition~\ref{SmaCuntz} implies that the operators $T_i:=T_{m_i}$ on $L^2(\T)$ form a Cuntz family with $T_0=S_m$, and $S_m$ is pure by Theorem~\ref{larryspurecond}. Thus the operator $S_1:L^2(\T)\oplus L^2(\T)\to L^2(\T)$ defined by $S_1(f,g)=T_1f+T_2g$ is a unitary isomorphism of $L:=L^2(\T)\oplus L^2(\T)$ onto the complement $(S_m(L^2(\T)))^\perp$, and the hypotheses of Proposition~\ref{genwavelet} are satisfied with $B=\{(1,0),(0,1)\}$ and $\rho=\mu\oplus\mu$. We deduce that
the set 
\[
\{U_1S_1(1,0),U_1S_1(0,1)\}=\{U_1T_11,U_1T_21\}=\{U_1m_1,U_1m_2\}
\] 
generates a wavelet basis
\[
\{S_\infty^{-j}\mu_\infty(k)U_1m_i:j\in\Z, k\in\Z, i=1,2\}
\]
for $L^2(\T)_\infty$. 

Applying the isomorphism $R_\infty$ gives an orthonormal basis
\[
\{D^{-j}\lambda_kR_\infty U_1m_i:j\in\Z, k\in\Z, i=1,2\}
\]
for $L^2(\CR,\nu)$. Let
\[
\psi_i(x)=R_\infty(U_1m_i)(x)=(R_1m_i)(x)=(D^{-1}Rm_i)(x)=2^{1/2}(Rm_i)(3x);
\]
in terms of the basis $e_n$ for $L^2(\T)$ used to define $R$ in Lemma~\ref{defR}, we have $m_1=e_1$ and $m_2=2^{-1/2}(e_0-e_2)$, so 
\begin{align*}
\psi_1(x)&=2^{1/2}\chi_{C+1}(3x)=2^{1/2}\chi_{3^{-1}(C+1)}(x),\text{ and }\\ 
\psi_2(x)&=2^{1/2}(2^{-1/2}\chi_C-2^{-1/2}\chi_{C+2})(3x)=\chi_{3^{-1}C}-\chi_{3^{-1}(C+2)}(x).
\end{align*}
Thus we recover the following theorem of Dutkay and Jorgensen \cite{dutjor42}:

\begin{theorem}\label{dutjorwavelet}
Let $\psi_1=2^{1/2}\chi_{3^{-1}(C+1)}$ and $\psi_2=\chi_{3^{-1}C}-\chi_{3^{-1}(C+2)}$. Then 
\[
\{\psi_{i,j,k}(x)=2^{j/2}\psi_i(3^jx-k):i=1,2,j\in\Z,k\in\Z\}
\]
is an orthonormal basis for $L^2(\CR,\nu)$. 
\end{theorem} 

\begin{example}
More generally, one can form a one-parameter family of multi-wavelets corresponding to dilation and translation on the filled-out Cantor set $\CR$.  For $r$ satisfying $|r|\leq 2^{-1/2}$ set $m_0(z)=2^{-1/2}(1+z^2)$, as above, and take
\begin{align*}
m_{1,r}(z)&:=-((1-2r^2)/2)^{1/2}+2^{1/2}rz+((1-2r^2)/2)^{1/2}z^2,\\
m_{2,r}(z)&:=r+(1-2r^2)^{1/2}z-rz^2.
\end{align*}
The remarks made in Example~\ref{exsoffilters}(c) imply that $\{m_0,m_{1,r},m_{2,r}\}$ is a filter bank, and the above argument shows that the pair 
\begin{align*}
\psi_{1,r}&:=-(1-2r^2)^{1/2}\chi_{3^{-1}C}+2r\chi_{3^{-1}(C+1)}+(1-2r^2)^{1/2}\chi_{3^{-1}(C+2)}\\
\psi_{2,r}&:=2^{1/2}\big(r\chi_{3^{-1}C}+(1-2r^2)^{1/2}\chi_{3^{-1}(C+1)}-r\chi_{3^{-1}(C+2)}\big)
\end{align*}
is a multi-wavelet for dilation by $3$ on $L^2({\mathcal R}, \nu)$; to recover Theorem~\ref{dutjorwavelet}, take $r=2^{-1/2}$.

There is also a version of Theorem~\ref{dutjorwavelet} which starts from the characteristic function of the Sierpinski gasket (see \cite{dmp}).
\end{example}

\section{Functions on solenoids}

Suppose that  $m:\widehat\Gamma\to \C$ is a filter for $\alpha^*\in \End\widehat\Gamma$. Then the representation $\mu:\Gamma\to U(L^2(\widehat\Gamma))$ defined by $(\mu_\gamma f)(\zeta)=\zeta(\gamma)f(\zeta)$ satisfies $S_m\mu_\gamma=\mu_{\alpha(\gamma)}S_m$. Thus the direct limit construction of \S\ref{abstractbs} gives a direct limit $(L^2(\widehat\Gamma)_\infty,U_n)$ together with a dilation $S_\infty$ and a representation $\mu_\infty$ of $\Gamma$ on $L^2(\widehat\Gamma)_\infty$ such that $S_\infty U_n=U_nS_m$ and $S_\infty\mu_\infty(\gamma)=\mu_\infty(\alpha(\gamma))S_\infty$. We want to identify this direct limit with an $L^2$-space of functions on the solenoid $\CS_{\alpha^*}:=\varprojlim(\widehat\Gamma,\alpha^*)$; this is motivated by previous work of Jorgensen \cite{Jor} and Dutkay \cite[\S5.2]{dut}, where $\Gamma=\Z$, $\alpha$ is multiplication by $N$, and $\CS_{\alpha^*}$ is the usual solenoid $\CS_N:=\varprojlim(\T, z\mapsto z^N)$. Then, as applications of our result, we will rederive a theorem of Dutkay on a ``Fourier transform'' for the Cantor set, and settle a question of Ionescu and Muhly about the support of the measure on the solenoid when $m$ is a low-pass filter.

To define the $L^2$-space on the solenoid, we need some background material on measures on solenoids. The first lemma is a modern formulation of a classical result (see, for example, \cite[Proposition~27.8]{parth}).

\begin{lemma}\label{abstractmeas}
Suppose that $r_n:T_{n+1}\to T_n$ is an inverse system of compact spaces with each $r_n$ surjective, and $\mu_n$ is a family of measures on $T_n$ such that $\mu_0$ is a probability measure and 
\begin{equation}\label{consistency}
\int (f\circ r_n)\,d\mu_{n+1}=\int f\,d\mu_n\ \text{ for $f\in C(T_n)$.}
\end{equation} 
Let $T_\infty=\varprojlim (T_n,r_n)$, and denote the canonical map from $T_\infty$ to $T_n$ by $\pi_n$. Then there is a unique probability measure $\mu$ on $T_\infty$ such that 
\[
\int (f\circ \pi_n)\,d\mu=\int  f\,d\mu_n\ \text{ for $f\in C(T_n)$.}
\]
\end{lemma} 

\begin{proof}
Since each $r_n$ is surjective, so is each $\pi_n$, and the map $\pi_n^*:f\mapsto f\circ \pi_n$ of $C(T_n)$ into $C(T_\infty)$ is isometric. The subset $\bigcup_{n=0}^\infty \pi_n^*(C(T_n))$ of $C(T_\infty)$ is a unital $*$-subalgebra of $C(T_\infty)$ which separates points of $T_\infty$, and hence by the Stone-Weierstrass theorem is dense in $C(T_\infty)$. Construct a functional $\phi$ on the dense subset $\bigcup \pi_n^*(C(T_n))$ of $C(T_\infty)$ by $\phi(\pi_n^*(f))=\int f\,d\mu_n$ for $f\in C(T_n)$; equation \eqref{consistency} implies that $\phi$ is well-defined. Taking $f=1$ in \eqref{consistency} shows that each $\mu_n$ is a probability measure; since the maps $\pi_n^*$ are isometric, this implies that $\phi$ is a positive functional with norm~$1$. Thus $\phi$ extends to a positive functional of norm $1$ on $C(T_\infty)$, and the Riesz representation theorem gives us the measure $\mu$. The uniqueness follows from density of $\bigcup_{n=0}^\infty \pi_n^*(C(T_n))$.
\end{proof}

Now we return to our specific situation, where  we again write $(K,\beta)$ for $(\widehat\Gamma,\alpha^*)$.

\begin{prop}\label{tauexists}
Denote by $\pi_n$ the canonical map of $\CS_\beta:=\varprojlim (K,\beta)$ onto the $n$th copy of $K$. There is a unique probability measure\footnote{When $K=\T$ and $\beta(z)=z^N$, this is same as the measure constructed by Dutkay in \cite[Proposition~4.2(i)]{dut}. In our notation, his defining property is 
\begin{equation}\label{dutkaysdef}
\int_{\CS_N} (f\circ\pi_n)\,d\tau=\int_{\T} \frac{1}{N^{n}}\Big(\sum_{w^{N^{n}}=z} f(w)\big(\textstyle{\prod_{j=0}^{n-1}|m(w^{N^j})|^2}\big)\Big)\,dz.
\end{equation} 
and his uniqueness statement is \cite[Proposition~4.2(ii)]{dut}. To see that our defining property is equivalent, notice that for any $g\in L^\infty(\T)$ and any $p\in \N$, we have
\begin{align*}
\int_{\T} \frac{1}{p}\Big(\sum_{w^{p}=z} g(w)\Big)\,dz
&=\int_{0}^1 \frac{1}{p}\Big(\sum_{j=0}^{p-1} g(e^{2\pi i(j+x)/p})\Big)\,dx\\
&=\sum_{j=0}^{p-1}\int_{j/p}^{(j+1)/p} g(e^{2\pi it})\,dt\\
&=\int_{\T} g(z)\,dz.
\end{align*}} 
 $\tau$ on $\CS_\beta$ such that for every $f\in C(K)$,
\begin{equation}\label{easierdef}
\int_{\CS_\beta} (f\circ\pi_n)\,d\tau=\int_{K} f(k)\big(\textstyle{\prod_{j=0}^{n-1}|m(\beta^j(k))|^2}\big)\,dk.
\end{equation}
\end{prop}

For the proof we need the following lemma, which follows from part (b) of Lemma~\ref{changevbles} by essentially the same calculation which proves that $S_m$ is an isometry (see \eqref{Smisometric}).

\begin{lemma}\label{changevble}
For every $g\in L^\infty(K)$ we have
\begin{equation*}\label{Sisometric}
\int_{K}g(\beta(k))|m(k)|^2\,dk=\int_K g(k)\,dk.
\end{equation*} 
\end{lemma}

\begin{proof}[Proof of Proposition~\ref{tauexists}]
We take $\tau_0$ to be normalized Haar measure, and define measures $\tau_n$ for $n\geq 1$ by
\begin{equation}\label{deftaun}
\int f\,d\tau_n=\int_{K} f(k)\big({\textstyle{\prod_{j=0}^{n-1}|m(\beta^j(k))|^2}}\big)\,dk \ \text{ for $f\in C(K)$.}
\end{equation}
To verify the consistency condition \eqref{consistency}, let $f\in C(K)$. Then
\begin{equation}\label{checkconsist}
\int(f\circ r_n)\,d\tau_{n+1}
=\int_{K} f(\beta(k))\big(\textstyle{\prod_{j=1}^{n}|m(\beta^j(k))|^2}\big)|m(k)|^2\,dk.
\end{equation}
Now Lemma~\ref{changevble} implies that the right-hand side of \eqref{checkconsist} is
\[
\int_{K} f(k)\big({\textstyle{\prod_{j=1}^{n}|m(\beta^{j-1}(k))|^2}}\big)\,dk=\int f\,d\tau_n.
\]
Thus the measures $\tau_n$ satisfy the hypotheses of Lemma~\ref{abstractmeas}, and the result follows from that lemma.
\end{proof}

We now want to identify the direct limit $(L^2(K)_\infty,U_n)$ with $(L^2(\CS_\beta,\tau),\pi_n^*)$. For this to be useful, we need to know what the isomorphism does to the dilation $S_\infty$ and the translations $\mu_\infty(\gamma)$. To describe the dilation on $L^2(\CS_\beta,\tau)$ we need the shift $h:\CS_\beta\to \CS_\beta$ characterized by $\pi_n(h(\zeta))=\pi_{n-1}(\zeta)$; if we realise elements of the inverse limit as sequences $\zeta=\{\zeta_n:n\geq 0\}$ satisfying $\beta(\zeta_{n+1})=\zeta_n$, then $h(\zeta_0,\zeta_1,\cdots)=(\beta(\zeta_0), \zeta_0, \zeta_1,\cdots)$.

\begin{theorem}\label{solenoidreal}
Suppose that $m:\widehat\Gamma\to \C$ is a filter for $\alpha^*\in \End\widehat\Gamma$ such that $m^{-1}(0)$ has Haar-measure zero. Let $\tau$ be the measure on $\CS_{\alpha^*}$ described in Proposition~\ref{tauexists}. Then there is an isomorphism $V_\infty$ of $L^2(\CS_{\alpha^*},\tau)$ onto the direct limit $L^2(\widehat\Gamma)_\infty=\varinjlim(L^2(\widehat\Gamma), S_m)$ such that
\begin{itemize}
\item[(a)] $V_\infty(g\circ\pi_n)=U_n\big(g\big(\prod_{j=0}^{n-1}(m\circ\alpha^{*j})\big)\big)$;
\smallskip
\item[(b)] $(V_\infty^*S_\infty V_\infty f)(\zeta)=m(\pi_0(\zeta))f(h(\zeta))$; and
\smallskip
\item[(c)] $(V_\infty^*\mu_\infty(\gamma) V_\infty f)(\zeta)=\pi_0(\zeta)(\gamma)f(\zeta)$.
\end{itemize}
\end{theorem}

We have chosen to look for an isomorphism from $L^2(\CS_{\alpha^*},\tau)$ to  $L^2(\widehat\Gamma)_\infty$ because this will be more convenient in the applications. However, this choice means that we cannot simply apply Theorem~\ref{iddirlim} to find the desired isomorphism. So we need to find different ways of exploiting the universal property of the direct limit.

\begin{proof}
Again we write $(K,\beta)$ for $(\widehat\Gamma,\alpha^*)$.
We begin by showing that the direct limit system defining $L^2(K)_\infty$, in which each Hilbert space is $L^2(K)$, is isomorphic to one in which the $n$th Hilbert space is $L^2(K,\tau_n)$ (where $\tau_n$ is the measure defined in \eqref{deftaun}). We define $T_n:L^2(K,\tau_n)\to L^2(K,\tau_{n+1})$ by $T_nf=f\circ \beta$; the consistency condition $\int(f\circ r_n)\,d\tau_{n+1}=\int f\,d\tau_n$ (checked in the proof of Proposition~~\ref{tauexists}) says that $T_n$ is an isometry. With $V_0=1$ and 
\[
V_nf:=\big({\textstyle \prod_{j=0}^{n-1}(m\circ\beta^j)}\big)f,
\]
we have the following commutative diagram of isometries: 
\[\xygraph{
{L^2(K)}="w0":[rr]{L^2(K,\tau_1)}="w1"^{T_0}:[rr]{L^2(K,\tau_2)}="w2"^{T_1}:[rr]{\cdots}="w3"^{T_2}
"w0":[uu]{L^2(K)}="v0"^{V_0}:[rr]{L^2(K)}="v1"^{S_m}:[rr]{L^2(K)}="v2"^{S_m}:[rr]{\cdots}="v3"^{S_m}
"w1":"v1"^{V_1}"w2":"v2"^{V_2}
}\]
Since the filter $m$ is non-zero except on a set of measure zero, each $V_n$ is surjective, and the $V_n$ form an isomorphism of the direct systems.

To identify the direct limit of the new system, we consider the maps $R_n:f\mapsto f\circ \pi_{n}$; equation \eqref{easierdef} implies that $R_n$ is an isometry of $L^2(K,\tau_n)$ into $L^2(\CS_\beta,\tau)$, and the formula $\beta\circ\pi_{n+1}=\pi_{n}$ implies that we have a commutative diagram
\[\label{commdiagforHilb}\xygraph{
{L^2(K)}="w0":[rr]{L^2(K,\tau_1)}="w1"^{T_0}:[rr]{L^2(K,\tau_2)}="w2"^{T_1}:[rr]{\cdots}="w3"^{T_2}
"w1":[dd]{L^2(\CS_\beta,\tau)}="v1"_{R_1}"w0":"v1"_{R_0}"w2":"v1"^{R_2}
}\]
Since the functions of the form $f\circ\pi_n$ span a dense subspace of $C(\CS_\beta)$ and hence also of $L^2(\CS_\beta,\tau)$, the isometries $R_n$ induce an isomorphism of the direct limit onto $L^2(\CS_\beta,\tau)$. Alternatively, we can say that $(L^2(\CS_\beta,\tau),R_n)$ is a direct limit for the system.

Since isomorphic direct systems have isomorphic direct limits, we deduce that there is an isomorphism $V_\infty$ of $L^2(\CS_\beta,\tau)$ onto $L^2(K)_\infty$ such that $V_\infty R_n=U_nV_n$, which is equation~(a).

It is enough to verify formulas (b) and (c) for $f$ of the form $f=R_ng=g\circ\pi_{n}$. For (b), we have
\[
V_\infty^*S_\infty V_\infty R_n=V_\infty^*S_\infty U_nV_n=(V_\infty^* U_n)(S_mV_n)=R_nV_n^*V_{n+1}T_n.
\]
To compute the latter, we let $g\in C(K)$ and $\zeta\in \CS_\beta$. Then
\begin{align*}
(R_nV_n^*V_{n+1}T_ng)(\zeta)
&=(V_n^*V_{n+1}T_ng)(\pi_{n}(\zeta))\\
&=\big(\textstyle{\prod_{j=0}^{n-1}m(\beta^j(\pi_{n}(\zeta)))^{-1}}\big)\big(\textstyle{\prod_{j=0}^{n}m(\beta^j(\pi_{n}(\zeta)))}\big)g(\beta(\pi_{n}(\zeta)))\\
&=m(\beta^n(\pi_{n}(\zeta)))g(\pi_{n-1}(\zeta))\\
&=m(\pi_0(\zeta))(R_ng)(h(\zeta)),
\end{align*}
and (b) follows. 

For (c), we begin by expanding
\[
V_\infty^*\mu_\infty(\gamma)V_\infty R_n=V_\infty^*\mu_\infty(\gamma)U_nV_n=V_\infty^*U_n\mu_{\alpha^n(\gamma)}V_n.
\]
Now we observe that both $\mu_{\alpha^n(\gamma)}$ and $V_n$ are multiplication operators, and hence commute (formally at least: strictly speaking, the two $\mu_{\alpha^n(\gamma)}$ act on different spaces). Thus 
\[
V_\infty^*\mu_\infty(\gamma)V_\infty R_n=V_\infty^*U_nV_n\mu_{\alpha^n(\gamma)}=R_n\mu_{\alpha^n(\gamma)}.
\]
For $g\in C(K)$ and $\zeta\in \CS_\beta$, we have
\begin{align*}
(R_n\mu_{\alpha^n(\gamma)}g)(\zeta)&=(\mu_{\alpha^n(\gamma)}g)(\pi_{n}(\zeta))=\pi_{n}(\zeta)(\alpha^n(\gamma))g(\pi_{n}(\zeta))\\
&=\beta^n(\pi_n(\zeta))(\gamma)(R_ng)(\zeta)=\pi_0(\zeta)(\gamma)(R_ng)(\zeta),
\end{align*}
which gives (c).
\end{proof}

\subsection{Dutkay's Fourier transform for $\CR$}

As a first application of Theorem~\ref{solenoidreal}, we apply it with $\Gamma=\Z$, $\alpha(j)=3j$ and $m(z)=2^{-1/2}(1+z^2)$. The resulting isometry $S_m$ on $L^2(\T)$ is the same one we considered in \S\ref{sec:cantor},  so Theorem~\ref{solenoidreal} gives an alternative realization of the direct limit $L^2(\T)_\infty$ as a space of functions on the solenoid $\CS_3$. Combining this isomorphism with that of Proposition~\ref{idwithL2CR} gives an isomorphism of $L^2(\CS_3,\tau)$ onto $L^2(\CR,\nu)$. The inverse of this isomorphism is Dutkay's ``Fourier transform for $\CR$'', as established in \cite[Corollary~5.8]{dut}.

\begin{cor}\label{dutkayFt}
Consider the filter $m(z)=2^{-1/2}(1+z^2)$ for dilation by $3$, and let $(L^2(\CR,\nu),D,\lambda)$ be as in \S\textnormal{\ref{sec:cantor}}. Let $\tau$ be the measure on the solenoid $\CS_3=\varprojlim(\T,z\mapsto z^3)$ described in Proposition~\textnormal{\ref{tauexists}}. Then there is an isomorphism $\CF$ of $L^2(\CR,\nu)$ onto $L^2(\CS_3,\tau)$ such that
\begin{itemize}
\item[(a)] $(\CF D\CF^* f)(\zeta)=m(\pi_0(\zeta))f(h(\zeta))$,
\smallskip
\item[(b)] $(\CF\lambda_k \CF^*f)(\zeta)=\pi_0(\zeta)^kf(\zeta)$, and
\smallskip
\item[(c)] $\CF(\chi_C)=1$.
\end{itemize}
\end{cor}

\begin{proof}
The composition of the isomorphism $V_\infty:L^2(\CS_3,\tau)\to L^2(\T)_\infty$ of Theorem~\ref{solenoidreal} with the isomorphism $R_\infty: L^2(\T)_\infty\to L^2(\CR,\nu)$ constructed in the proof of Proposition~\ref{idwithL2CR} is an isomorphism of $L^2(\CS_3,\tau)$ onto $L^2(\CR,\nu)$; we take $\CF:=(R_\infty\circ V_\infty)^*$. Then (a) and (b) follow from the properties of $R_\infty$ and $V_\infty$. For (c), we compute
\[
R_\infty V_\infty(1)=R_\infty V_\infty(1\circ \pi_0)=R_\infty(U_0(1))=R_0(1)=\chi_C.\qedhere
\]
\end{proof}

Dutkay's proof of Corollary~\ref{dutkayFt} uses a uniqueness theorem for a family of ``wavelet representations'' of the Baumslag-Solitar group $\Z[N^{-1}]\rtimes\Z$ due to Jorgensen \cite[Theorem~2.4]{Jor}. In the next section we show that Jorgensen's theorem also follows easily from our Theorem~\ref{iddirlim}.

Corollary~\ref{dutkayFt} and  Theorem~\ref{dutjorwavelet} imply that the functions 
\[
\hat{\psi_1}=2^{1/2}\CF(\chi_{3^{-1}(C+1)})\ \text{ and }\ \hat{\psi_2}=\CF(\chi_{3^{-1}C}-\chi_{3^{-1}(C+2)})
\]
generate a wavelet basis for $L^2(\CS_3,\tau)$ with respect to the dilation described in (a) and the translation described in (b).

\subsection{The winding line}

When $m:\T\to \C$ is a low-pass filter for dilation by $N$ and $m^{-1}(0)$ has measure zero, we can identify the direct limit $\varinjlim (L^2(\T),S_m)$ with either $L^2(\R)$ (as in Example~\ref{dilatmatrix}) or $L^2(\CS_N,\tau)$ (using Theorem~\ref{solenoidreal}). Combining these two results gives an isomorphism $R_\infty\circ V_\infty$ of $L^2(\CS_N,\tau)$ onto $L^2(\R)$, from which we will obtain a completely different description of the measure $\tau$ as Lebesgue measure on a ``winding line'' obtained from an embedding of $\R$ in the solenoid. 

We begin by deriving a formula for $R_\infty\circ V_\infty$ on functions of the form $g\circ \pi_n$. We resume the notation of Example~\ref{dilatmatrix}, and define $D_N:L^2(\R)\to L^2(\R)$ by $(D_Nf)(t)=N^{1/2}f(Nt)$. Then part (a) of Theorem~\ref{solenoidreal} gives
\begin{align*}
R_\infty\circ V_\infty(g\circ\pi_n)(x)
&=R_\infty\circ U_n\big(z\mapsto g(z)\big(\textstyle{\prod_{j=0}^{n-1}}m(z^{N^{j}})\big)\big)(x)\\
&=D_N^{-n}R\big(z\mapsto g(z)\big(\textstyle{\prod_{j=0}^{n-1}}m(z^{N^{j}})\big)\big)(x)\\
&=N^{-n/2}g(e^{2\pi iN^{-n}x})\big(\textstyle{\prod_{j=0}^{n-1}}m(e^{2\pi iN^{-n+j}x})\big)\phi(N^{-n}x),
\end{align*}
and $n$ applications of the scaling identity \eqref{scaling1} imply that
\[
R_\infty\circ V_\infty(g\circ\pi_n)(x)=g(e^{2\pi iN^{-n}x})\phi(x).
\]
So we introduce the function $w:\R\to \CS_N$ which is uniquely characterized by 
\begin{equation}\label{defwinding}
\pi_n(w(x))=e^{2\pi iN^{-n}x}\ \text{ for $x\in \R$ and $n\geq 0$;}
\end{equation} 
this is the ``winding line'' referred to above.

\begin{theorem}\label{windingline}
Suppose that $m:\T\to \C$ is a low-pass filter for dilation by $N$ which is Lipschitz near $1$, which satisfies Cohen's condition, and for which $m^{-1}(0)$ has measure zero. Let $\phi\in L^2(\R)$ be the associated scaling function satisfying \eqref{scalingdef}, \eqref{scaling1} and \eqref{scaling2}. Let $w:\R\to \CS_N$ be the function satisfying \eqref{defwinding}. Then the measure $\tau$ of Proposition~\ref{tauexists} satisfies
\begin{equation}\label{windingmeas}
\int_{\CS_N} f\,d\tau=\int_{\R} f(w(x))|\phi(x)|^2\,dx\ \text{ for $f\in C(\CS_N)$,}
\end{equation}
and the formula $(Tf)(x):=f(w(x))\phi(x)$ defines a unitary isomorphism $T$ of $L^2(\CS_N,\tau)$ onto $L^2(\R)$ such that $T(V_\infty^*S_\infty V_\infty)T^*=D_N$ and $T(V_\infty^*\mu_\infty(k) V_\infty)T^*$ is multiplication by $e^{2\pi ikx}$. 
\end{theorem}

\begin{proof}
We fix $g\in C(\T)$, $n\geq 0$, and compute:
\begin{align*}
\int_{\R} (g\circ \pi_n)(w(x))&|\phi(x)|^2\,dx=\int_{\R} g(e^{2\pi iN^{-n}x})|\phi(x)|^2\,dx\\
&=\int_{\R} g(e^{2\pi is})N^{n}|\phi(N^ns)|^2\,ds\\
&=\int_{\R} g(e^{2\pi is})\big(\textstyle{\prod_{j=0}^{n-1}|m(e^{2\pi iN^js})|^2}\big)|\phi(s)|^2\,ds\quad \text{ (using \eqref{scaling1})}\\
&=\sum_{k\in\Z}\int_0^1 g(e^{2\pi is})\big(\textstyle{\prod_{j=0}^{n-1}|m(e^{2\pi iN^js})|^2}\big)|\phi(s+k)|^2\,ds\\
&=\int_{\T} g(z)\big(\textstyle{\prod_{j=0}^{n-1}|m(z^{N^j})|^2}\big)\,dz\quad\text{ (using \eqref{scaling2})}\\
&=\int (g\circ \pi_n)\,d\tau\quad\text{  (by \eqref{easierdef}).}
\end{align*}
We can now deduce \eqref{windingmeas} from the uniqueness in Proposition~\ref{tauexists}. Equation~\eqref{windingmeas} implies that $T$ is an isometry of $L^2(\CS_N,\tau)$ into $L^2(\R)$; surjectivity will be easy after we have the other properties of $T$.

For the last two assertions, we let $f\in L^2(\CS_N,\tau)$. First, we use part~(b) of Theorem~\ref{solenoidreal} to see that
\begin{align*}
(T(V_\infty^*S_\infty V_\infty)f)(x)
&=(V_\infty^*S_\infty V_\infty f)(w(x))\phi(x)\\
&=m(\pi_0(w(x)))f(h(w(x)))\phi(x)\\
&=m(e^{2\pi ix})f(w(Nx))\phi(x),
\end{align*}
which by the scaling equation is $N^{1/2}\phi(Nx)f(w(Nx))=(D_NTf)(x)$. Next, we use part~(c) of Theorem~\ref{solenoidreal} to see that
\begin{align*}
(T(V_\infty^*\mu_\infty(k) V_\infty)f)(x)
&=(V_\infty^*\mu_\infty(k) V_\infty f)(w(x))\phi(x)\\
&=\pi_0(w(x))^kf(w(x))\phi(x)\\
&=e^{2\pi ikx}(Tf)(x).
\end{align*}

We still have to prove that $T$ is surjective. For $f\in L^2(\T)$, we have $T(f\circ\pi_0)(x)=f(e^{2\pi ix})\phi(x)$, so the range of $T$ contains the subspace 
\[
V_0=\clsp\{x\mapsto e^{2\pi ikx}\phi(x):k\in \Z\}
\]
in the usual multiresolution analysis $\{V_j\}$ for $L^2(\R)$ associated to the low-pass filter $m$ for dilation by $N$ (as in Example~\ref{dilatmatrix}). Since the formula $T(V_\infty^*S_\infty V_\infty)=D_NT$ implies that the range of $T$ is closed under dilation, the range of $T$ is a closed subspace containing $\bigcup_jV_j$, and hence must be all of $L^2(\R)$.
\end{proof}

\begin{remark}
Ionescu and Muhly \cite{im} have also recognised that the direct limit $L^2(\T)_\infty$ can be realised as both $L^2(\R)$ and $L^2(\CS_N,\tau)$, and conjectured that the measure $\tau$ is supported on the winding line and is absolutely continuous with respect to the measure pulled over from Lebesgue measure on $\R$ (see the second last paragraph of \cite{im}). The formula \eqref{windingmeas} confirms this conjecture, and also identifies the Radon-Nikodym derivative in terms of the scaling function $\phi$.
\end{remark}

\begin{remark}
Theorem~\ref{windingline} holds without significant change for any dilation matrix $A$ and low-pass filter $m:\T^n\to \C$ satisfying the hypotheses of Example~\ref{dilatmatrix}. In this case $A:\R^n\to \R^n$ induces an endomorphism $\alpha$ of $\T^n=\R^n/\Z^n$, and the theorem gives an embedding $w$ of $\R^n$ round the solenoid $\CS_A:=\varprojlim(\T^n,\alpha)$ which carries the measure $|\phi(x)|^2\,dx$ into $\tau$. 
\end{remark}

\section{Uniqueness of the wavelet representation}

We let $(\Gamma_\infty,\iota^n)$ denote the direct limit $\varinjlim(\Gamma,\alpha)$, and write $\alpha_\infty$ for the automorphism of $\Gamma_\infty$ characterized by $\alpha_\infty\circ\iota^n=\iota^n\circ\alpha$. We identify $\Gamma$ with the subgroup $\iota^0(\Gamma)$ of $\Gamma_\infty$, so that $\alpha=\alpha_\infty|_\Gamma$. The semidirect product $\BS(\Gamma,\alpha):=\Gamma_\infty\rtimes_{\alpha_\infty}\Z$ is known as the \emph{Baumslag-Solitar group} of $\alpha$ (see, for example, \cite{dj}). Unitary representations $W:\BS(\Gamma,\alpha)\to U(H)$ are determined by a unitary representation 
$T=W|_\Gamma$ and a unitary operator $U=W_{(0,1)}$ satisfying $UT_\gamma=T_{\alpha(\gamma)}U$; we recover $W$ as $W_{(\alpha_\infty^{-n}(\gamma),j)}=U^{-n}T_\gamma U^{n+j}$. Associated to the unitary representation $T$ is a representation $\pi_W:C(\widehat\Gamma)\to B(H)$ which takes the functions $\widehat\gamma:\omega\mapsto \omega(\gamma)$ to the operators $T_\gamma$; the pair $(\pi_W,U)$ is then covariant in the sense that $U\pi_W(f)U^*=\pi_W(f\circ \alpha^*)$. 

Now suppose that $m$ is a filter for $\alpha^*$ and $h:\widehat\Gamma\to [0,\infty)$ is an integrable function such that
\[
\frac{1}{N}\sum_{a\in \ker\alpha^*} |m(a\omega)|^2h(\omega)=h(\alpha^*(\omega))\ \text{ for almost all $\omega\in \widehat\Gamma$.}
\]
In this section we suppose that $m$ is a continuous function (but see Remark~\ref{normality?} below). Following \cite{Jor}, we say that a unitary representation $W$ of $BS(\Gamma,\alpha)$ on $H$ is a \emph{wavelet representation for $m$ with correlation function $h$} if there is a cyclic vector $\phi\in H$ such that
\begin{itemize}
\item[]\begin{itemize}\item[(WR1)] $U\phi=\pi_W(m)\phi$, and
\smallskip
\item[(WR2)] $(T_\gamma\phi\,|\,\phi)=\int_{\widehat\Gamma} \omega(\gamma)h(\omega)\,d\omega$ for every $\gamma\in \Gamma$;
\end{itemize}
\end{itemize}
we then call $\phi$ a \emph{scaling element} for $W$. Notice that if $h=1$, then (WR2) says that the set $\{T_\gamma\phi:\gamma\in\Gamma\}$ is orthonormal, so in general the correlation function is a measure of the extent to which this set is not orthonormal.

\begin{example} We define a measure $\sigma$ on $\widehat\Gamma$ by $\int f\,d\sigma=\int_{\widehat\Gamma} f(\omega)h(\omega)\,d\omega$, and then a routine calculation, as in \cite[Lemma~3.2]{Jor}, shows that the operator $S_m$ is isometric on $L^2(\widehat\Gamma,\sigma)$. Applying the construction of \S\ref{abstractbs} to $S_m$ and the representation $\mu$ defined by $\mu_\gamma:f\mapsto \widehat\gamma f$ gives a direct limit
$(L^2(\widehat\Gamma,\sigma)_\infty, U_n)$, a unitary dilation $S_\infty$
of $S_m$, and a representation $\mu_\infty$ of $\Gamma$ on
$L^2(\widehat\Gamma,\sigma)_\infty$ such that $S_\infty U_n=U_nS_m$ and
$S_\infty\mu_\infty(\gamma)=\mu_\infty(\alpha(\gamma))S_\infty$. This last identity says that $(S_\infty, \mu_\infty)$
determines a unitary representation $W$ of the Baumslag-Solitar group $\BS(\Gamma,\alpha)$ on $L^2(\widehat\Gamma,\sigma)_\infty$, which we claim is a wavelet representation for $m$ and $h$.

First note that the elements $\mu_\infty(\gamma)U_0(1)=U_0(\mu_\gamma(1))=U_0\widehat\gamma$ span a dense subset of $U_0(L^2(\widehat\Gamma,\sigma))$.
Since $S_\infty^{-n}$ maps the range of $U_0$
onto the range of $U_n$, it follows that the
elements $S_{\infty}^{-n}\mu_\infty(\gamma)U_0(1)=W_{(\alpha_\infty^{-n}(\gamma), n)}U_0(1)$ span a dense subspace of $L^2(\widehat\Gamma,\sigma)_\infty$, and hence $\phi:=U_0(1)$ is cyclic. To verify (WR1), notice that both sides are continuous in $m$, and so it suffices to consider $m=\sum_{\gamma \in \Gamma}a_\gamma \widehat\gamma$. Then
\begin{align*}
\pi_W(m)U_0(1)
 &=\sum_{\gamma\in \Gamma}a_\gamma\mu_\infty(\gamma)U_0(1)=\sum_{\gamma\in \Gamma}a_\gamma
 U_0(\widehat\gamma)=U_0\big({\textstyle{\sum_{\gamma\in \Gamma}a_\gamma\widehat\gamma}}\big)\\
 &=U_0(m)=U_0S_m(1)=S_\infty U_0(1).
\end{align*}
For (WR2), we compute 
\[(\pi_W(\widehat\gamma)U_0(1)\,|\,U_0(1))=(\mu_\infty(\gamma)U_0(1)\,|\,U_0(1))=(U_0(\widehat\gamma)\,|\,U_0(1))=(\widehat\gamma\,|\,1),
\]
which is the right-hand side of (WR2).
\end{example}

In the previous example, we have basically summarized the discussion in \cite[pages~15--20]{Jor} under slightly different hypotheses (see Remark~\ref{normality?}). The next result is the analogue of uniqueness in \cite[Theorem~2.4]{Jor}, and our proof differs from the original in its use of the universal property via Theorem~\ref{iddirlim}.

\begin{prop}[Jorgensen]
Suppose that $W:\BS(\Gamma,\alpha)\to U(H)$ is a wavelet representation for $m$ with correlation function $h$ and scaling element $\phi$. Then there is an isomorphism $X$ of $L^2(\widehat\Gamma,\sigma)_\infty$ onto $H$ such that 
\begin{itemize}
\item[(a)] $W_{(\gamma,0)}=X\mu_\infty(\gamma)X^*$ for $\gamma\in \Gamma$,
\smallskip
\item[(b)] $W_{(0,1)}=XS_\infty X^*$, and
\smallskip
\item[(c)] $XU_0(1)=\phi$.
\end{itemize}
\end{prop}

\begin{proof}
We aim to apply Theorem~\ref{iddirlim} with $\lambda_\gamma=W_{(\gamma,0)}$ and $D=W_{(0,1)}$. We define $R:C(\widehat\Gamma)\to H$ by $Rf=\pi_W(f)\phi$, and claim that $R$ extends to an isometry on $L^2(\widehat\Gamma,\sigma)$. Since $\sigma$ is a regular Borel measure, $C(\widehat\Gamma)$ is dense in $L^2(\widehat\Gamma,\sigma)$, and it suffices to check that $\|Rf\|^2=\|f\|^2$ for $f$ of the form $f=\sum c_\gamma\widehat\gamma$. This follows from a straightforward calculation using the equality in (WR2) above.

The relation $D\lambda_\gamma D^*=\lambda_{\alpha(\gamma)}$ is the covariance relation which characterizes the representations of $\BS(\Gamma,\alpha)$. The covariance of $(\pi_W,D)=(\pi_W,W_{(0,1)})$ implies that
\begin{align*}
(RS_m)f&=\pi_W(m(f\circ\alpha^*))\phi=\pi_W(f\circ\alpha^*)\pi_W(m)\phi\\
&=\pi_W(f\circ\alpha^*)D\phi=D\pi_W(f)\phi=(DR)f,
\end{align*}
and hence $RS_m=DR$. Since $\mu_\gamma(f)$ is the pointwise product $\widehat\gamma f$ we have
\[
(R\mu_\gamma)f=R(\widehat\gamma f)=\pi_W(\widehat\gamma f)\phi=\pi_W(\widehat\gamma)(\pi_W(f)\phi)=W_{(\gamma,0)}(Rf)=(\lambda_\gamma R)f,
\]
and $R\mu_\gamma=\lambda_\gamma R$. So Theorem~\ref{iddirlim} gives an isomorphism $R_\infty$ of $L^2(\widehat\Gamma,\sigma)_\infty$ onto the closure of $\bigcup_{n=0}^\infty D^{-n}R\big(L^2(\widehat\Gamma,\sigma)\big)$. The range of $R$ contains every $\lambda_\gamma(\phi)=R(\widehat\gamma)$, and every $D^n\lambda_\gamma(\phi)=R(S_m\widehat\gamma)$ with $n>0$, so the cyclicity of $\phi$ implies that $R_\infty$ is surjective. 

Properties (a) and (b) of $X:=R_\infty$ follow from the properties of $R_\infty$ in Theorem~\ref{iddirlim}. For (c), notice that $XU_0(1)=R_\infty U_0(1)=R(1)=\phi$, as required.
\end{proof}

\begin{remark}\label{normality?}
When $\Gamma=\Z$ and $\alpha(j)=Nj$, we recover a characterization of the wavelet representations of the classical Baumslag-Solitar group $\Z[N^{-1}]\rtimes \Z$. This is slightly different from Theorem~2.4 of \cite{Jor}, since we have assumed that $m$ is continuous. The result in \cite{Jor} applies to Borel filters $m$, but requires an extra hypothesis on the representation $W$ which ensures that the representation $\pi_W$ of $C(\T)$ extends to a normal representation of $L^\infty(\T)$, so that one can make sense of $\pi_W(m)$ in such a way that the covariance of $(\pi_W,U)$ is preserved. It is not immediately obvious that when $m(z)=2^{-1/2}(1+z^2)$, the representation $W$ of $\Z[3^{-1}]\rtimes\Z$ on $L^2(\CR,\nu)$ constructed in \S\ref{sec:cantor} satisfies this normality hypothesis, so the above version of \cite[Theorem~2.4]{Jor} may be better suited to the application in \cite[\S5.2]{dut}.
\end{remark}

\section*{Conclusions}

We have tackled a variety of problems associated with multiresolution analyses and wavelets using a systematic approach based on direct limits of Hilbert spaces and their universal properties. Previous authors have observed the connection with direct limits (often referring to them as ``inductive limits'', and often referring to the process of turning an isometry into a unitary as ``dilation''); the innovation in our approach lies in the systematic use of the universal property to identify a particular direct limit with a concrete Hilbert space of functions, such as $L^2(\R)$ or $L^2(\CS_N)$. This approach does not eliminate the need for analytic arguments, but it does seem to help identify exactly what analysis is needed: in each situation we have considered, we have quickly been able to identify the ingredients necessary to make our approach work.

\end{document}